\documentclass[11pt,reqno]{amsart}
\usepackage{amsmath,amssymb,amsthm,amsfonts,amscd,color,flafter,graphicx,verbatim,pinlabel,mathrsfs,setspace}
\usepackage{mathtools}
\usepackage{enumitem}
\usepackage{hyperref}
\usepackage{graphicx}
\usepackage{subcaption}
\usepackage{float}
\usepackage{abstract}
\usepackage{cleveref}
\usepackage{microtype}

\newtheorem{theorem}{Theorem}[section]
\newtheorem{lemma}[theorem]{Lemma}
\newtheorem{proposition}[theorem]{Proposition}
\newtheorem{corollary}[theorem]{Corollary}
\newtheorem{question}{Question}

\theoremstyle{definition}
\newtheorem{definition}[theorem]{Definition}

\theoremstyle{remark}
\newtheorem{remark}[theorem]{Remark}

\usepackage{placeins}

\makeatletter
\renewcommand{\@settitle}{%
  \begin{center}
    \Large\bfseries\boldmath
    \@title\par
  \end{center}%
}
\makeatother

\title{On the Concordance Invariants of Cable Links}
\author{Aadhishwaran Bhuvaraghamoorthy}
\date{}

\begin{document}
\raggedbottom
\maketitle

\begin{abstract}

We study the Beliakova--Wehrli link $s$-invariant and Cavallo's
link $\tau$-invariant of coherently oriented cables. For knots
satisfying extremal four-ball genus and maximal self-linking
hypotheses, we establish bounds for the $s$-invariant of
$(p,q)$-cables, where $p$ and $q$ are not necessarily coprime,
and show that these bounds are sharp under certain conditions.
For the $\tau$-invariant, we prove analogous bounds for arbitrary
companion knots and exact formulas whenever
$\tau(K)=\pm g_4(K)\ne 0$. Specializing to the coherently oriented
cable links associated to an $n$-framed knot, we apply
these bounds to the lasagna $s$-invariant of Ren and Willis.
The resulting applications provide upper bounds for the lasagna
$s$-invariant of knot traces and a criterion for distinguishing
homeomorphic zero traces up to orientation-preserving
diffeomorphism.

\end{abstract}

\section{Introduction}
\label{sec:introduction}

Over the last decade, concordance invariants derived from Khovanov homology have taken a prominent role in the study of exotic phenomena in dimension four \cite{Piccirillo20,ManolescuPiccirillo23,MWW22}. This paper grew out of our attempt to compute $s$-invariants of cable links with an eye towards computing lasagna $s$-invariants of compact smooth $4$-manifolds with boundary. The latter were introduced by Ren and Willis in \cite{RW25} building on work of Manolescu and Neithalath \cite{MN22}, who gave an explicit description of skein lasagna modules for $2$-handlebodies in terms of cabled Khovanov--Rozansky homology. As such, computations of the lasagna $s$-invariant require computing $s$-invariants of certain cables of $n$-framed knots with prescribed orientations on each component.

Computing the $s$-invariant of cable links with arbitrary
orientations on their components is generally difficult. Ren
\cite{Ren23} computed these invariants for torus links with
arbitrary orientations, but results for general companion
knots remain limited. In this article, we use the link
cobordism inequalities of Beliakova and Wehrli \cite{BW08}
and Cavallo \cite{CavalloTau} to extend Van Cott's cabling
framework \cite{VanCott10} to the link $s$- and
$\tau$-invariants of coherently oriented cables.

Given an oriented knot $K\subset S^3$, Van Cott studied in
\cite{VanCott10} the Rasmussen and Ozsváth--Szabó concordance
invariants of $(p,q)$-cables when $\gcd(p,q)=1$. Using
Hedden's work \cite{MH09}, she obtained bounds on the
$\tau$-invariant of cables. Subsequently, Hom \cite{Hom2014}
strengthened these results. Computations of Rasmussen invariants of cable and satellite knots are also available in several special settings.
Wu \cite{Wu12} used simplified Khovanov--Rozansky complexes
to study $(2,2k+1)$-cables, while Lewark and Zibrowius
\cite{LewarkZibrowius24} obtained formulas for the Rasmussen
invariant over $\mathbb{F}_2$ for satellites with patterns
of wrapping number two. General formulas for the
characteristic-zero Rasmussen invariant of cables with
arbitrary longitudinal winding number remain limited. Our main result establishes Van Cott-type bounds for the
Rasmussen invariant of coherently oriented cable links
under some conditions on the smooth four-ball genus
$g_4(K)$ and the maximal self-linking number
$\overline{\operatorname{sl}}(K)$.

\subsection*{Notation}
Let $K\subset S^3$ be an oriented knot. We denote its
$(p,q)$-cable link by $K_{p,q}$, where $p>0$ is the
longitudinal winding number and $q\in\mathbb{Z}$ is the
meridional winding number, measured with respect to the
preferred longitude and meridian of $K$. Throughout this
paper, $K_{p,q}$ carries the coherent orientation induced
by $K$. It has $\gcd(p,q)$ components and is a knot precisely
when $\gcd(p,q)=1$.

For an oriented link $L\subset S^3$, let $r(L)$ denote the
link obtained by reversing the orientations of all its
components, and let $m(L)$ denote its mirror. We write
\[
-L\coloneqq m(r(L))
\]
for its orientation-reversed mirror. For a knot $K$, we
also use the notation $\overline{K}=m(K)$. Finally,
$\overline{\operatorname{sl}}(K)$ denotes the maximal
self-linking number among all transverse representatives
of the oriented knot type $K$ in
$(S^3,\xi_{\mathrm{std}})$.

For knots, Rasmussen's $s$-invariant \cite{Ras10} and the
Ozsv\'ath--Szab\'o $\tau$-invariant \cite{Ozsvath2003}
are smooth concordance invariants satisfying
\[
|s(K)| \leq 2g_4(K)
\qquad\text{and}\qquad
|\tau(K)| \leq g_4(K).
\]
Throughout this paper, we use the extensions to oriented
links introduced by Beliakova and Wehrli \cite{BW08}
and Cavallo \cite{CavalloTau}, respectively.

\begin{theorem}
\label{thm:main-s-cable}
Let $K\subset S^3$ be an oriented knot, and equip each cable $K_{p,q}$ with the coherent orientation induced by $K$. Suppose that one of the following conditions holds:
\[
\mathrm{(i)} \qquad
s(K)=2g_{4}(K)
\quad\text{and}\quad
s(K)=\overline{\operatorname{sl}}(K)+1,
\]
or
\[
\mathrm{(ii)} \qquad
s(K)=-2g_{4}(K)
\quad\text{and}\quad
s(\overline{K})=\overline{\operatorname{sl}}(\overline{K})+1.
\]
Then, for every integer $p>0$ and every integer $q$,
\[
p\,s(K)+(p-1)(q-1)
\leq
s(K_{p,q})
\leq
p\,s(K)+(p-1)(q+1).
\]
Moreover, if $K$ satisfies $\mathrm{(i)}$ and $q>0$, then
\[
s(K_{p,q})
=
p\,s(K)+(p-1)(q-1),
\]
whereas if $K$ satisfies $\mathrm{(ii)}$ and $q<0$, then
\[
s(K_{p,q})
=
p\,s(K)+(p-1)(q+1).
\]
\end{theorem}

As an immediate consequence, Theorem~\ref{thm:main-s-cable} applies to
quasi-positive knots and to their mirrors.

\begin{corollary}
\label{cor:quasi-positive-cables}
Let $K$ be a quasi-positive knot. Then the
bounds of Theorem~\ref{thm:main-s-cable} hold for $K$ and $\overline{K}$, and 
\[
s(K_{p,q})
=
p\,s(K)+(p-1)(q-1),
\]
for $q>0$, while
\[
s(\overline{K}_{p,q})
=
p\,s(\overline{K})+(p-1)(q+1).
\]
for $q<0$.
\end{corollary}

\begin{proof}
If $K$ is quasi-positive, then
\[
s(K)=2g_4(K)
\qquad\text{and}\qquad
\overline{\operatorname{sl}}(K)=s(K)-1;
\]
see \cite{Shumakovitch07,Plamenevskaya06}. Thus $K$ satisfies
hypothesis (i) of Theorem~\ref{thm:main-s-cable}. Meanwhile, since
\[
s(\overline{K})=-s(K)
\qquad\text{and}\qquad
g_4(\overline{K})=g_4(K),
\]
we have
\[
s(\overline{K})=-2g_4(\overline{K}).
\]
Since the mirror of $\overline{K}$ is $K$, the required
maximal self-linking condition is precisely
\[
s(K)=\overline{\operatorname{sl}}(K)+1,
\]
which holds by quasi-positivity. Hence $\overline{K}$
satisfies hypothesis (ii) of
Theorem~\ref{thm:main-s-cable}.
\end{proof}

\begin{remark}
\label{rem:quasi-positive-examples}
Since every positive knot is strongly quasi-positive \cite{Rudolph99},
Corollary~\ref{cor:quasi-positive-cables} applies, in particular, to
strongly quasi-positive knots, positive knots, positive braid knots, and
positive torus knots, together with their mirrors. Notice that the bounds hold for every integer $q$, whereas the upper and lower bounds are attained in the cases when $q>0$ for quasi-positive knots and when $q<0$ for their mirrors.
\end{remark}

For the $\tau$-invariant, no hypothesis on the companion knot is needed for the bounds. Combining the knot-cabling results of Van Cott \cite{VanCott10} and Hom \cite{Hom2014} with Cavallo's link cobordism inequalities \cite{CavalloTau}, we obtain the following extension to coherently oriented cables.

\begin{theorem}
\label{tau-main-thm}
Let $K\subset S^3$ be an oriented knot. For a positive integer
$p$ and an integer $q$, let $K_{p,q}$ denote the coherently
oriented $(p,q)$-cable of $K$, and set $d=\gcd(p,q)$. Here, $\tau$ denotes Cavallo's link invariant, which agrees with the Ozsváth--Szabó invariant on knots. Then
\[
p\,\tau(K)+\frac{(p-1)(q-1)+d-1}{2}
\le \tau(K_{p,q})
\le p\,\tau(K)+\frac{(p-1)(q+1)+d-1}{2}.
\]
Moreover, if $\tau(K)=g_4(K)>0$, then
\[
\tau(K_{p,q})
=
p\,\tau(K)+\frac{(p-1)(q-1)+d-1}{2}.
\]
If $\tau(K)=-g_4(K)<0$, then
\[
\tau(K_{p,q})
=
p\,\tau(K)+\frac{(p-1)(q+1)+d-1}{2}.
\]
\end{theorem}

The unrestricted bounds and equality statements for $\tau$ suggest the following questions for the link $s$-invariant.

\begin{question}
\label{ques:general-s-cabling-bounds}
Let $K\subset S^3$ be an oriented knot, let $p$ be a positive
integer, and let $q\in\mathbb{Z}$. Equip the cable $K_{p,q}$
with the coherent orientation induced by $K$. Do the bounds
\[
p\,s(K)+(p-1)(q-1)
\leq s(K_{p,q})
\leq p\,s(K)+(p-1)(q+1)
\]
hold for every knot $K$, without any additional hypotheses?
\end{question}

\begin{question}
\label{ques:general-s-cabling-equality}
Let $K\subset S^3$ be an oriented knot, let $p$ be a positive
integer, and let $q\in\mathbb{Z}$. Equip the cable $K_{p,q}$
with the coherent orientation induced by $K$. Can the
self-linking hypotheses and the sign restrictions on $q$
be removed from the equality statements of
Theorem~\ref{thm:main-s-cable}? More precisely, if $s(K)=2g_4(K)>0$, is it always true that
\[
s(K_{p,q})=p\,s(K)+(p-1)(q-1)?
\]
Similarly, if $s(K)=-2g_4(K)<0$, is it always true that
\[
s(K_{p,q})=p\,s(K)+(p-1)(q+1)?
\]
\end{question}

To connect these cabling results with knot traces, we use the notation $K(a,b)$ introduced by Manolescu and Neithalath \cite{MN22} for parallel copies of a framed knot. If $K$ is an oriented knot equipped with framing $n$, then $K(a,b)$ is obtained by taking $a+b$ parallel copies of $K$ determined by the chosen framing, with $a$ copies oriented as $K$ and $b$ copies oriented oppositely. Each component represents the longitude determined by the framing and hence winds once in the longitudinal direction of a tubular neighborhood of $K$. In the cabling convention used throughout this paper, the underlying unoriented link is therefore the $(a+b,n(a+b))$-cable link. We note that Manolescu and Neithalath use the opposite ordering of the cabling parameters, so the same underlying unoriented link is denoted the $(n(a+b),a+b)$-cable in \cite{MN22}.

\begin{figure}[!htbp]
    \centering
    \includegraphics[width=0.4\textwidth]{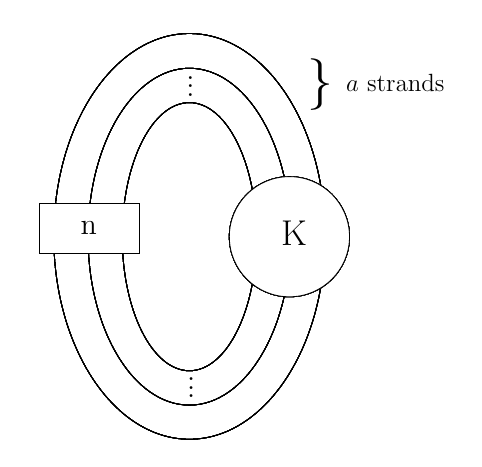}
    \caption{The oriented cable link $K(a,0)$ associated to an $n$-framed knot $K$. The $a$ parallel strands are coherently oriented, and the box labeled $n$ denotes $n$ full twists. In our convention, $K(a,0)$ is the $(a,na)$-cable link.}
    \label{fig:framed-cable}
\end{figure}

\begin{definition}
\label{def:K-a0}
Following the notation of Manolescu and Neithalath
\cite{MN22}, let $K\subset S^3$ be an $n$-framed oriented knot and let $a\geq 1$. We denote by $K(a,0)$ the corresponding oriented $(a,na)$-cable link, with all $a$ components oriented coherently with $K$.
\end{definition}

\begin{lemma}
\label{prop:s-framed-cable}
Let $K \subset S^{3}$ be an oriented knot equipped with framing $n\in\mathbb{Z}$, and let $K(a,0)$ denote the corresponding oriented cable link. Suppose that $K$ satisfies one of the hypotheses of Theorem~\ref{thm:main-s-cable}. Then, for every integer $a>0$,
\[
a\,s(K)+(a-1)(na-1)
\leq
s(K(a,0))
\leq
a\,s(K)+(a-1)(na+1).
\]
Moreover:
\begin{enumerate}
\item If $s(K)=2g_{4}(K)$ and $s(K)=\overline{\operatorname{sl}}(K)+1$, then, for $n>0$,
\[
s(K(a,0))=a\,s(K)+(a-1)(na-1).
\]
\item If $s(K)=-2g_{4}(K)$ and $s(\overline{K})=\overline{\operatorname{sl}}(\overline{K})+1$, then, for $n<0$,
\[
s(K(a,0))=a\,s(K)+(a-1)(na+1).
\]
\end{enumerate}
\end{lemma}

\begin{proof}
The result follows immediately from Theorem~\ref{thm:main-s-cable} by setting
$p=a$ and $q=na$.
\end{proof}

The bounds for the oriented cable links $K(a,0)$ also have applications to the lasagna $s$-invariant of knot traces. We first give a criterion for distinguishing homeomorphic zero traces up to diffeomorphism.

\begin{corollary}
\label{cor:exotic-zero-traces}
Let $K$ and $K'$ be $0$-framed oriented knots such that
\[
s(K)=2g_4(K)=\overline{\operatorname{sl}}(K)+1=2
\qquad\text{and}\qquad
s(K')<2,
\]
and suppose that there is an orientation-preserving homeomorphism from $X_0(K)$ to $X_0(K')$. If there exists an integer $M>0$ such that
\[
s\bigl(K(2r+1,0)\bigr)
\leq
-s\bigl(-K(1+r,r)\bigr)
\]
for every $r\geq M$, then $X_0(K)$ and $X_0(K')$ form an exotic pair of oriented zero traces.
\end{corollary}

Although Corollary~\ref{cor:exotic-zero-traces} concerns the zero-framed case, the cable bounds yield a related eventual comparison for a broader range of framings.

\begin{proposition}
\label{prop:eventual-comparison}
Let $K$ be an oriented knot equipped with framing $n$, and suppose that $K$ satisfies one of the hypotheses of Theorem~\ref{thm:main-s-cable}. Consider the sequences
\[
\left\{s\bigl(K(2r+1,0)\bigr)\right\}_{r\geq0}
\qquad\text{and}\qquad
\left\{-s\bigl(-K(1+r,r)\bigr)\right\}_{r\geq0}.
\]
In each of the following cases, there exists an integer $M>0$ such that, for every $r\geq M$,
\[
s\bigl(K(2r+1,0)\bigr)
\leq
-s\bigl(-K(1+r,r)\bigr):
\]
\begin{enumerate}
\item $n=0$ and $s(K)<0$;
\item $n<0$, with no additional restriction on $s(K)$.
\end{enumerate}
\end{proposition}

Finally, we obtain an upper bound for the lasagna $s$-invariant on knot traces.

\begin{proposition}
\label{prop:lasagna-bound}
Let $K\subset S^{3}$ be an oriented knot equipped with framing $n$, and suppose that $K$ satisfies one of the hypotheses of Theorem~\ref{thm:main-s-cable}. Let $X_n(K)$ denote the $n$-trace of $K$, and identify
\[
H_2(X_n(K))\cong\mathbb{Z}
\]
using the standard generator. Then, for every integer $\alpha>0$,
\[
s\bigl(X_n(K);\alpha\bigr)
\leq
\alpha\bigl(s(K)-n\bigr).
\]
\end{proposition}

\subsection*{Organization}
Section~\ref{sec:s-cables} extends Van Cott's monotonicity argument to coherently oriented cables and proves Theorem~\ref{thm:main-s-cable} and Proposition~\ref{prop:eventual-comparison}. Section~\ref{sec:lasagna} applies the bounds of Lemma~\ref{prop:s-framed-cable} to prove Corollary~\ref{cor:exotic-zero-traces} and Proposition~\ref{prop:lasagna-bound}. Finally, Section~\ref{sec:tau-cables} introduces a component correction to Van Cott's function and proves Theorem~\ref{tau-main-thm}.

\medskip

\subsection*{Acknowledgements}
The author thanks his advisor, \c{C}a\u{g}atay Kutluhan, for his
continued guidance, helpful discussions, valuable suggestions,
and encouragement throughout this work. The author also thanks
Qiuyu Ren for helpful discussions and feedback on an earlier
draft, and for pointing out that Van Cott's monotonicity theorem
could be extended to coherently oriented cable links. Further thanks are due to Cornelia A.
Van Cott for her helpful feedback and suggestions on the first
draft.

The author used ChatGPT in preparation of this
manuscript to assist with editing and expository improvement, as well as with literature searches. 

\section{Rasmussen Invariants of Coherently Oriented Cables}
\label{sec:s-cables}

The proof of Theorem~\ref{thm:main-s-cable} has two steps. First, we extend Van Cott's monotonicity theorem from cable knots to coherently oriented cable links using the link $s$-invariant cobordism inequality. Second, transverse cabling provides exact values for sufficiently positive cable knots, which determine bounds for all cabling parameters. We then construct surfaces bounding the cable
links to establish the equality statements. We conclude the section by constructing a cobordism to prove
Proposition~\ref{prop:eventual-comparison}. 

Let \(\nu\) be a concordance invariant satisfying the following three properties:
\begin{enumerate}
\item \(\nu\) is a homomorphism from the smooth knot concordance group
\(\mathcal{C}\) to \(\mathbb{Z}\).
\item For every knot \(K\),
\[
|\nu(K)|\leq g_{4}(K),
\]
where \(g_{4}(K)\) denotes the smooth \(4\)-ball genus of \(K\).
\item For every positive torus knot \(T_{p,q}\),
\[
\nu(T_{p,q})
=
\frac{(p-1)(q-1)}{2},
\]
where \(p,q\geq 1\) are relatively prime.
\end{enumerate}

A concordance invariant satisfying these properties is called a
\emph{slice-torus invariant}. In particular, both \(\tau\) and \(s/2\)
are slice-torus invariants.

Fix a knot $K$ and an integer \(p>0\). Following Van Cott \cite{VanCott10}, define the function
\[
h\colon
\{q\in\mathbb{Z}\mid \gcd(p,q)=1\}
\longrightarrow
\frac{1}{2}\mathbb{Z}
\]
by
\[
h(q)
=
\nu(K_{p,q})
-
\frac{p-1}{2}\,q.
\]
Van Cott proved the following result concerning the behavior of \(h\).

\begin{lemma}[Theorem~2 of \cite{VanCott10}]
\label{lem:vancott-monotonicity}
The function \(h\) is non-increasing and bounded from below. More precisely,
if \(q>r\), where both \(q\) and \(r\) are relatively prime to \(p\), then
\[
-(p-1)
\leq
h(q)-h(r)
\leq
0.
\]
\end{lemma}

For $\nu=s/2$, we can extend the preceding lemma to all integer cabling parameters. The key input is the following cobordism inequality for the oriented link $s$-invariant.

\begin{lemma}[Beliakova-Wehrli \cite{BW08}; Theorem~1.5 of \cite{MMSW23}]
\label{lem:mmsw-cobordism}
Let $\Sigma\subset I\times S^3$ be an oriented cobordism from $L_1$ to $L_2$. If every connected component of $\Sigma$ has a boundary component in $L_1$, then
\[
s(L_2)-s(L_1)\ge \chi(\Sigma).
\]
\end{lemma}

When every component of the cobordism meets both ends, applying the preceding lemma in both directions results in the following corollary which we will use throughout the rest of the section.

\begin{corollary}
\label{cor:s-cobordism}
Suppose that $\Sigma \subset I \times S^3$ is an oriented cobordism from a link $L$ to a link $L'$ such that every connected component of $\Sigma$ has boundary components in both $L$ and $L'$. Then their $s$-invariants satisfy
\[
\left|s(L')-s(L)\right| \leq -\chi(\Sigma).
\]
\end{corollary}

\begin{proof}
Applying Lemma~\ref{lem:mmsw-cobordism} with $L_1=L$ and $L_2=L'$, we obtain
\[
s(L')-s(L)\geq \chi(\Sigma).
\]
Next, reverse the direction of the cobordism by turning $\Sigma$ upside down. The reversed cobordism satisfies the same boundary hypothesis, so Lemma~\ref{lem:mmsw-cobordism} applies again, now with $L_1=L'$ and $L_2=L$. Thus
\[
s(L)-s(L')\geq \chi(\Sigma).
\]
Equivalently,
\[
s(L')-s(L)\leq -\chi(\Sigma).
\]
Combining this inequality with the first one yields
\[
\chi(\Sigma)\leq s(L')-s(L)\leq -\chi(\Sigma),
\]
and hence
\[
\left|s(L')-s(L)\right|\leq -\chi(\Sigma),
\]
as desired.
\end{proof}

\begin{remark}
\label{rem:s-genus-bound}
Let $L \subset S^3$ be an oriented link with $l$ components, and let
$F \subset B^4$ be a connected, smoothly embedded, oriented surface
with boundary $L$ and genus $g(F)=g_4(L)$. Here, $g_4(L)$ denotes the minimum genus among connected, smoothly embedded, oriented surfaces in $B^4$ with boundary $L$. Removing a small open
$4$-ball centered at an interior point of $F$ yields a connected
cobordism $\Sigma \subset S^3 \times I$ from $L$ to the unknot $U$, with
\[
\chi(\Sigma)=1-2g_4(L)-l.
\]
By Corollary~\ref{cor:s-cobordism},
\[
|s(L)-s(U)| \le -\chi(\Sigma)=2g_4(L)+l-1.
\]
Since $s(U)=0$, we obtain
\[
|s(L)| \le 2g_4(L)+l-1.
\]
\end{remark}

\begin{definition}
\label{def:cable-h}
Fix an oriented knot $K\subset S^3$ and a positive integer $p$. For each $q\in\mathbb{Z}$, equip the
cable $K_{p,q}$ with the coherent orientation induced by $K$.
Define
\[
h\colon \mathbb{Z}\longrightarrow \frac{1}{2}\mathbb{Z}
\]
by
\[
h(q)=\frac{s(K_{p,q})}{2}-\frac{p-1}{2}\,q.
\]
This extends Van Cott's definition of $h$ for the invariant $s/2$
to coherently oriented cables, without requiring $p$ and $q$
to be relatively prime.
\end{definition}

\begin{lemma}
\label{lem:vancott-monotonicity-2}
The function \(h\) is non-increasing and bounded from below. More precisely,
for any integers $q>r$,
\[
-(p-1)
\leq
h(q)-h(r)
\leq
0.
\]
\end{lemma}

\begin{proof}

If $p=1$, then $h(q)=s(K)/2$ for every $q$, so the result is immediate. Assume $p>1$, and let $q,r\in\mathbb{Z}$ with $q>r$.
View $K_{p,q}$ as the satellite of the companion knot $K$ with
pattern given by the closure of the braid $\delta^q$ embedded inside a solid torus, where $\delta=\sigma_1\cdots\sigma_{p-1}$. Identify this solid torus with a tubular neighborhood of $K$ using the preferred longitude. All band attachments described below take place within this neighborhood and respect the coherent orientations.

When $q> r\ge 0$, smoothing the crossings in $q-r$ copies of
$\delta$ transforms the pattern $\delta^q$ into $\delta^r$.
When $0\ge q>r$, reverse the analogous smoothing cobordism from
$\delta^r$ to $\delta^q$. If $q>0>r$, concatenate the smoothing
cobordism from $\delta^q$ to the identity braid with the reverse
smoothing cobordism from the identity braid to $\delta^r$.
Each oriented smoothing corresponds to a band attachment, and
each copy of $\delta$ or $\delta^{-1}$ contains $p-1$ crossings.
Thus, in every case, we obtain an oriented cobordism $\Sigma$
from $K_{p,q}$ to $K_{p,r}$ consisting of $(p-1)(q-r)$ band
attachments. Since these bands include a full block of adjacent braid generators, they connect all $p$ strands. Consequently, $\Sigma$ is connected, and
\[
\chi(\Sigma)=-(p-1)(q-r).
\]
By Corollary~\ref{cor:s-cobordism},
\[
s(K_{p,q})-s(K_{p,r})
\le -\chi(\Sigma)
=(p-1)(q-r).
\]
It follows that
\[
h(q)-h(r)
=
\frac{s(K_{p,q})-s(K_{p,r})}{2}
-\frac{p-1}{2}(q-r)
\le 0.
\]
Hence, $h$ is non-increasing.

To establish the lower bound, choose integers $Q$ and $R$, both
relatively prime to $p$, such that $Q>q>r>R$. By monotonicity,
\[
h(Q)\le h(q)\le h(r)\le h(R),
\]
therefore
\[
h(q)-h(r)\ge h(Q)-h(R).
\]
Since $K_{p,Q}$ and $K_{p,R}$ are knots, Van Cott's inequality gives
\[
h(Q)-h(R)\ge -(p-1).
\]
Combining these inequalities, we obtain
\[
-(p-1)\le h(q)-h(r)\le 0,
\]
which also implies that $h$ is bounded from below.
\end{proof}

\begin{proof}[Proof of Theorem~\ref{thm:main-s-cable}]

Suppose that $p>1$, since the case $p=1$ is trivial. We begin by
constructing surfaces bounding the cable links for both
positive and negative cabling parameters.

Let $q\ne 0$, set $d=\gcd(p,q)$, and choose a connected,
smoothly embedded, oriented surface $F\subset B^4$ with
boundary $K$ and genus $g_4(K)$. Take $p$ disjoint parallel
copies of $F$, whose boundaries form $p$ parallel zero-framed copies of $K$. In a collar of the boundary, attach one oriented band
for each crossing of the braid
\[
(\sigma_1\cdots\sigma_{p-1})^q.
\]
The bands have the sign prescribed by $q$, and their attachment
turns the boundary into the coherently oriented cable $K_{p,q}$.
Since $\sigma_1\cdots\sigma_{p-1}$ is a cyclic permutation,
the bands connect all $p$ parallel surfaces. Therefore, the
resulting surface $F_{p,q}$ is connected and has $d$ boundary
components. Moreover,
\[
\chi(F_{p,q})
=
p\,\chi(F)-(p-1)|q|
=
p\bigl(1-2g_4(K)\bigr)-(p-1)|q|.
\]
Writing $\chi(F_{p,q})=2-2g(F_{p,q})-d$, we obtain
\begin{equation}
\label{eq:cable-surface-genus}
g_4(K_{p,q})
\le g(F_{p,q})
=
p\,g_4(K)
+\frac{(p-1)(|q|-1)-(d-1)}{2}.
\end{equation}
We first assume that
\[
s(K)=2g_4(K)
\qquad\text{and}\qquad
s(K)=\overline{\operatorname{sl}}(K)+1.
\]
Choose a transverse representative $\mathcal{T}$ of $K$
realizing the maximal self-linking number. Thus,
\[
\operatorname{sl}(\mathcal{T})=s(K)-1.
\]
Let $N>0$ be relatively prime to $p$ and sufficiently large that
\[
\frac{N}{p}>\lceil\omega(K)\rceil,
\]
where $\omega(K)$ denotes the contact width appearing in
\cite{CEM24}. Since $p>1$ and $\gcd(p,N)=1$, the slope $N/p$ is not an integer. By the transverse cabling results of Chakraborty,
Etnyre, and Min \cite{CEM24}, taking $k=0$, there exists a transverse
representative $\mathcal{T}_{p,N}$ of $K_{p,N}$ satisfying
\[
\operatorname{sl}(\mathcal{T}_{p,N})
=
p\,\operatorname{sl}(\mathcal{T})+(p-1)N;
\]
see the definition preceding Theorem~1.8 and Theorem~1.8
of \cite{CEM24}.
The $s$-Bennequin inequality
\cite[Proposition~4]{Plamenevskaya06} then gives
\begin{align*}
s(K_{p,N})
&\ge \operatorname{sl}(\mathcal{T}_{p,N})+1\\
&=p\bigl(s(K)-1\bigr)+(p-1)N+1\\
&=p\,s(K)+(p-1)(N-1).
\end{align*}
On the other hand, since $\gcd(p,N)=1$, the cable $K_{p,N}$
is a knot. Applying \eqref{eq:cable-surface-genus} and the
genus bound for the knot $s$-invariant yields
\begin{align*}
s(K_{p,N})
&\le 2g_4(K_{p,N})\\
&\le 2p\,g_4(K)+(p-1)(N-1)\\
&=p\,s(K)+(p-1)(N-1).
\end{align*}
Consequently, for all sufficiently large positive integers $N$ relatively prime to $p$, we have
\begin{equation}
\label{eq:large-positive-cable-s}
s(K_{p,N})=p\,s(K)+(p-1)(N-1).
\end{equation}

Next, consider the function $h$ from Definition~\ref{def:cable-h}.
Equation~\eqref{eq:large-positive-cable-s} gives
\[
h(N)=\frac{p\,s(K)-(p-1)}{2}.
\]
For any $q\in\mathbb{Z}$, choose such an $N$ with $N>q$.
By Lemma~\ref{lem:vancott-monotonicity-2},
\[
h(N)\le h(q)\le h(N)+(p-1).
\]
Substituting the definitions of $h(N)$ and $h(q)$ and
rearranging, we obtain
\[
p\,s(K)+(p-1)(q-1)
\le s(K_{p,q})
\le p\,s(K)+(p-1)(q+1).
\]
Suppose now that $q>0$, and set $d=\gcd(p,q)$.
By Remark~\ref{rem:s-genus-bound} and \eqref{eq:cable-surface-genus},
\begin{align*}
s(K_{p,q})
&\le 2g_4(K_{p,q})+d-1\\
&\le 2p\,g_4(K)+(p-1)(q-1)\\
&=p\,s(K)+(p-1)(q-1).
\end{align*}
Combining this with the lower bound above gives
\[
s(K_{p,q})=p\,s(K)+(p-1)(q-1).
\]
This proves the theorem under the first hypothesis.

We now assume that
\[
s(K)=-2g_4(K)
\qquad\text{and}\qquad
s(\overline{K})
=\overline{\operatorname{sl}}(\overline{K})+1.
\]
Since
\[
s(\overline{K})=-s(K)=2g_4(\overline{K}),
\]
the mirror $\overline{K}$ satisfies the first hypothesis.
For every sufficiently large positive integer $N$ relatively
prime to $p$, the first part of the proof gives
\[
s(\overline{K}_{p,N})
=
p\,s(\overline{K})+(p-1)(N-1).
\]
Because
\[
\overline{K_{p,-N}}\cong\overline{K}_{p,N}
\]
and these cables are knots, symmetry of the $s$-invariant for knots under mirroring implies
\[
s(K_{p,-N})
=
p\,s(K)-(p-1)(N-1).
\]
Therefore,
\[
h(-N)=\frac{p\,s(K)+(p-1)}{2}.
\]
Given $q\in\mathbb{Z}$, choose such an $N$ with $-N<q$.
By Lemma~\ref{lem:vancott-monotonicity-2},
\[
h(-N)-(p-1)\le h(q)\le h(-N).
\]
Substituting the definitions and rearranging yields
\[
p\,s(K)+(p-1)(q-1)
\le s(K_{p,q})
\le p\,s(K)+(p-1)(q+1).
\]
Finally, suppose that $q<0$, and set $d=\gcd(p,q)$.
Remark~\ref{rem:s-genus-bound} and \eqref{eq:cable-surface-genus} give
\begin{align*}
s(K_{p,q})
&\ge -2g_4(K_{p,q})-(d-1)\\
&\ge -2p\,g_4(K)-(p-1)(|q|-1)\\
&=p\,s(K)+(p-1)(q+1),
\end{align*}
where the last equality uses $s(K)=-2g_4(K)$ and $q<0$.
Together with the upper bound above, we have
\[
s(K_{p,q})=p\,s(K)+(p-1)(q+1).
\]
This completes the proof.
\end{proof}

For use in the subsequent proofs, we construct band cobordisms
between $K(a,0)$ and its neighboring cable knots.

Let $K\subset S^3$ be an oriented knot with framing $n$, and
let $a$ be a positive integer. Regard $K(a,0)$ as the coherently
oriented $(a,na)$-cable link of $K$. For $t=\pm1$, the cable
$K_{a,na+t}$ is a knot because $\gcd(a,na+t)=1$. Attaching
$a-1$ oriented half-twisted bands between successive parallel
strands gives a connected cobordism
\[
\Sigma_t\colon K(a,0)\longrightarrow K_{a,na+t},
\qquad
\chi(\Sigma_t)=-(a-1).
\]
The coherent orientations allow all bands to have the same
sign. Positive half-twists produce $K_{a,na+1}$, while negative
half-twists produce $K_{a,na-1}$, as illustrated in
Figures~\ref{fig:2a} and~\ref{fig:2b}, respectively.

\begin{figure}[!htbp]
\centering

\begin{subfigure}{0.55\textwidth}
    \centering
    \includegraphics[width=\textwidth]{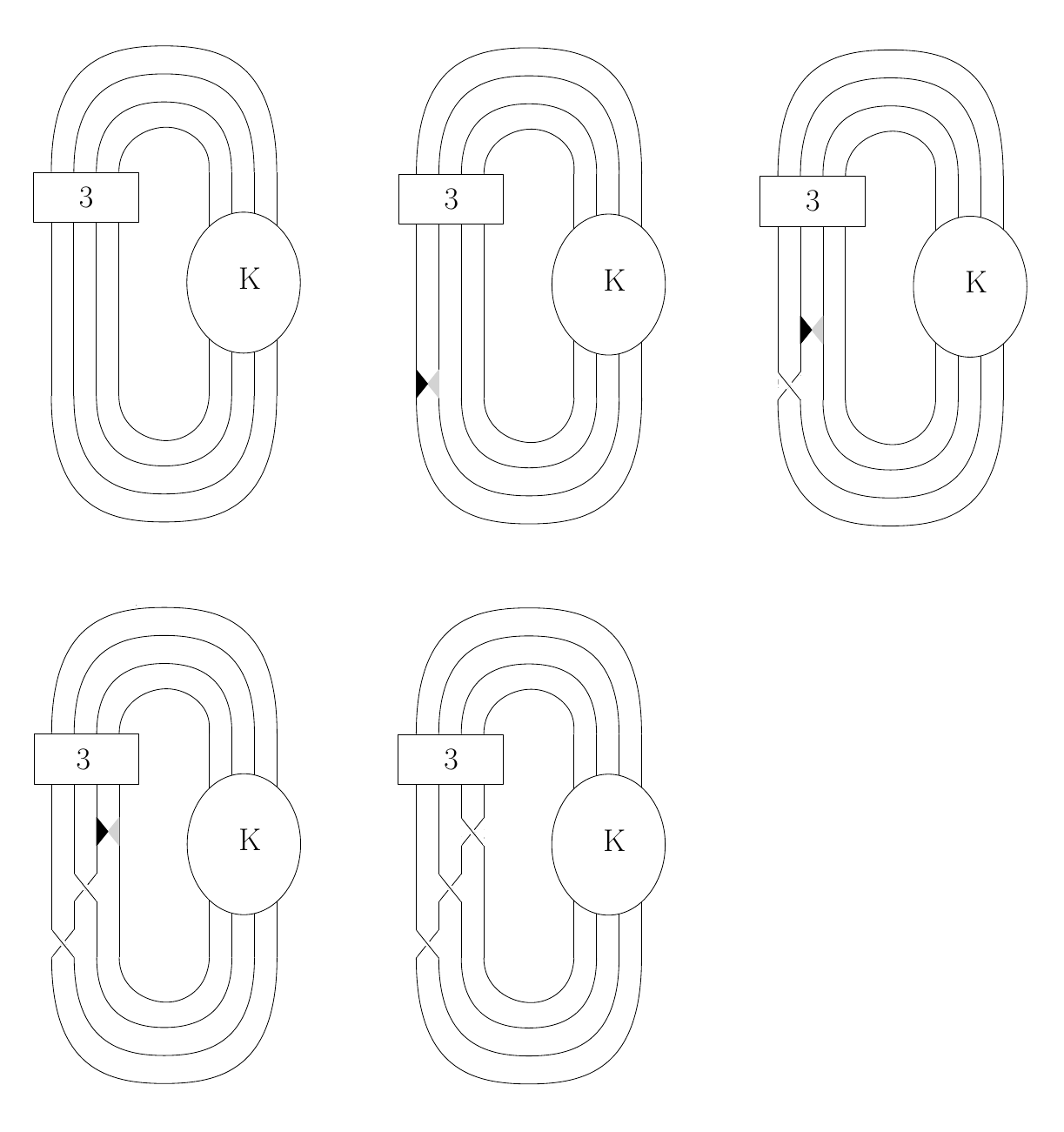}
    \caption{For framing $n=3$, the cobordism from $K(4,0)$ to $K_{4,13}$ obtained by attaching positive half-twisted bands.}
    \label{fig:2a}
\end{subfigure}

\vspace{0.5cm}

\begin{subfigure}{0.55\textwidth}
    \centering
    \includegraphics[width=\textwidth]{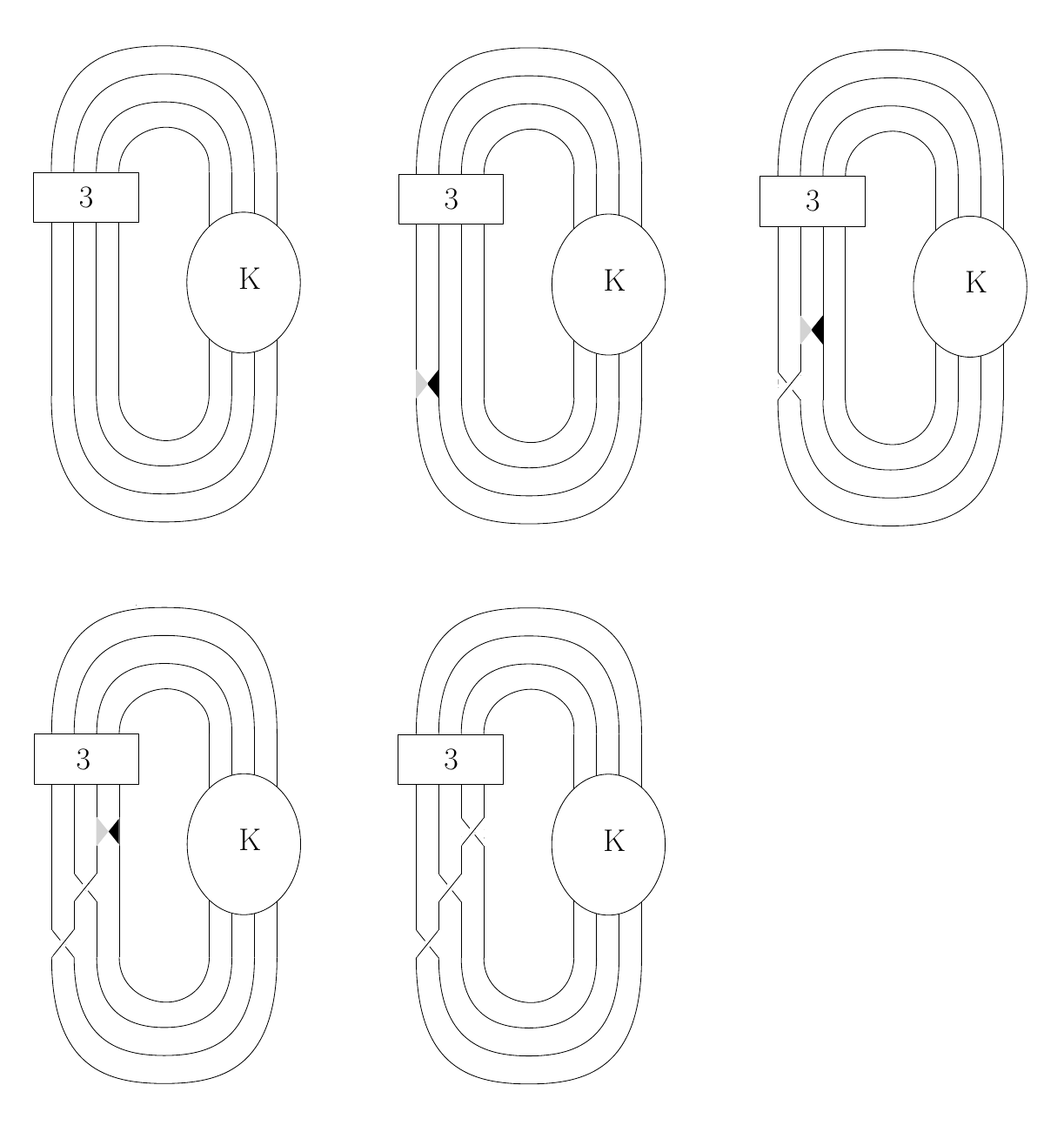}
    \caption{For framing $n=3$, the cobordism from $K(4,0)$ to $K_{4,11}$ obtained by attaching negative half-twisted bands.}
    \label{fig:2b}
\end{subfigure}

\caption{Band cobordisms from $K(a,0)$ to the neighboring cable knots $K_{a,na\pm1}$.}
\label{fig:2}

\end{figure}

\begin{proof}[Proof of Proposition~\ref{prop:eventual-comparison}]
Let $K$ be an oriented knot equipped with framing $n$, and suppose that $K$ satisfies one of the two conditions of Theorem~\ref{thm:main-s-cable}. For $r\geq 0$, consider the links
\[
K(2r+1,0)
\qquad\text{and}\qquad
-K(1+r,r).
\]
We construct a cobordism from $K(2r+1,0)\sqcup -K(1+r,r)$ to $K_{2r+1,n(2r+1)-1}\sqcup -K$.

First, let $\Sigma_1$ be the connected cobordism from $K(2r+1,0)$ to
$K_{2r+1,n(2r+1)-1}$ obtained by attaching $2r$ negative half-twisted bands, as in Figure~\ref{fig:2b}. The choice $t=-1$ gives the required upper bound in Theorem~\ref{thm:main-s-cable}. Next, construct a cobordism $\Sigma_2$ from $-K(1+r,r)$ to $-K$ as follows. Pair the $r$ oppositely oriented parallel pairs and push the annuli that they cobound into the interior of $I\times S^3$; carry the remaining component to the outgoing end by a product annulus. Thus $\Sigma_2$ is the disjoint union of $r$ annuli whose boundaries lie on the incoming end and one product annulus meeting both ends. In particular, every component of $\Sigma_2$ has an incoming boundary component and $\chi(\Sigma_2)=0$. Therefore,
\[
\Sigma=\Sigma_1\sqcup\Sigma_2
\]
is a cobordism from $K(2r+1,0)\sqcup -K(1+r,r)$ to $K_{2r+1,n(2r+1)-1}\sqcup -K$. Since $\chi(\Sigma_1)=-2r$ and $\chi(\Sigma_2)=0$, we have
\[
\chi(\Sigma)=-2r.
\]
Hence, by Lemma~\ref{lem:mmsw-cobordism},
\[
s\bigl(K_{2r+1,n(2r+1)-1}\sqcup -K\bigr)
-
s\bigl(K(2r+1,0)\sqcup -K(1+r,r)\bigr)
\geq -2r.
\]
In the convention of \cite{MMSW23}, Proposition~7.6(4) gives
\[
s(L_1\sqcup L_2)=s(L_1)+s(L_2)-1
\]
for non-empty links $L_1$ and $L_2$. Applying this formula to both ends of the cobordism, the two constant terms cancel, and we obtain
\[
s\bigl(K_{2r+1,n(2r+1)-1}\bigr)
+s(-K)
-s\bigl(K(2r+1,0)\bigr)
-s\bigl(-K(1+r,r)\bigr)
\geq -2r.
\]
Applying the upper bound from Theorem~\ref{thm:main-s-cable} to
$K_{2r+1,n(2r+1)-1}$ gives
\[
s\bigl(K_{2r+1,n(2r+1)-1}\bigr)
\leq
(2r+1)s(K)+2rn(2r+1).
\]
Since $s(-K)=-s(K)$, it follows that
\[
2r\,s(K)+2rn(2r+1)
-s\bigl(K(2r+1,0)\bigr)
-s\bigl(-K(1+r,r)\bigr)
\geq -2r.
\]
Equivalently,
\[
-s\bigl(-K(1+r,r)\bigr)
\geq
s\bigl(K(2r+1,0)\bigr)
-2r-2r\,s(K)-2rn(2r+1).
\]
We now consider the two cases separately. Suppose first that $n=0$ and $s(K)<0$. Since the Rasmussen invariant of a knot is even-valued, $s(K)\leq -2$. Therefore,
\[
-2r-2r\,s(K)
=
-2r\bigl(1+s(K)\bigr)
\geq 0.
\]
Hence, for every $r\geq 0$,
\[
-s\bigl(-K(1+r,r)\bigr)
\geq
s\bigl(K(2r+1,0)\bigr).
\]
Now suppose that $n<0$. The additional term can be written as
\[
-2r-2r\,s(K)-2rn(2r+1)
=
2r\bigl(-1-s(K)-n(2r+1)\bigr).
\]
Since $n<0$, this quantity tends to $+\infty$ as $r\to\infty$. Thus, there exists an integer $M>0$ such that, for every $r\geq M$,
\[
-2r-2r\,s(K)-2rn(2r+1)\geq 0.
\]
Consequently,
\[
-s\bigl(-K(1+r,r)\bigr)
\geq
s\bigl(K(2r+1,0)\bigr)
\]
for every $r\geq M$. This completes the proof.
\end{proof}

\section[Applications to the Lasagna s-Invariant]{Applications to the Lasagna $s$-Invariant}
\label{sec:lasagna}

We now apply Theorem~\ref{thm:main-s-cable} to knot traces. Ren and Willis \cite{RW25} express the lasagna $s$-invariant of a $2$-handlebody as a limit involving oriented cables of its attaching link. We recall this formula, then use it to prove Corollary~\ref{cor:exotic-zero-traces} and Proposition~\ref{prop:lasagna-bound}.

Suppose $X$ is a $2$-handlebody obtained by attaching $2$-handles to $B^4$ along a framed oriented link $K=K_1\cup\cdots\cup K_m\subset S^3$, and let $L\subset\partial X$ be represented by a framed oriented link in $S^3$ disjoint from $K$. Since $X$ is obtained by attaching $m$ $2$-handles, we have $H_2(X)\cong\mathbb{Z}^m$. For $\alpha=(\alpha_1,\alpha_2,\ldots,\alpha_m)\in\mathbb{Z}^m$, define
\[
\alpha^+=(\alpha_1^+,\alpha_2^+,\ldots,\alpha_m^+)
\qquad\text{and}\qquad
\alpha^-=(\alpha_1^-,\alpha_2^-,\ldots,\alpha_m^-),
\]
where $\alpha_i^+=\max\{0,\alpha_i\}$ and $\alpha_i^-=\max\{0,-\alpha_i\}$. Write $|\alpha|=\sum_i |\alpha_i|$ and, for $r=(r_1,\ldots,r_m)\in\mathbb{Z}_{\geq0}^m$, write $|r|=\sum_i r_i$. We denote by $w$ the total writhe of the framed link. This is the linking number of the link with its push-off in the framing direction. Ren and Willis use $-L$ to denote the mirror of $L$. Since the oriented link $s$-invariant is unchanged by simultaneously reversing all component orientations \cite[Proposition~7.6(1)]{MMSW23}, their formula agrees with our convention $-L=m(r(L))$. With this notation, their result takes the following form.

\begin{lemma}[Ren--Willis \cite{RW25}]
\label{lem:ren-willis-formula}
Let $X$ and $L$ be as above. For any $\alpha=(\alpha_1,\alpha_2,\ldots,\alpha_m)\in\mathbb{Z}^m$, we have
\[
\begin{aligned}
s(X;L;\alpha)
={}&-\lim_{r\to\infty}
\left(
s\bigl(-(K(\alpha^++r,\alpha^-+r)\cup L)\bigr)
+2|r|
\right)\\
&-w\bigl(K(\alpha^+,\alpha^-)\cup L\bigr)
-|\alpha|+1.
\end{aligned}
\]
where $r=(r_1,\ldots,r_m)\in\mathbb{Z}_{\geq 0}^m$ and $r\to\infty$ means that each $r_i\to\infty$. Moreover, the quantity appearing inside the limit is non-decreasing with respect to $r$.
\end{lemma}

\begin{remark}
\label{rem:trace-formula}
When the attaching link $K$ is a knot, $L=\varnothing$, and $\alpha=1$, the $2$-handlebody $X$ is precisely the $n$-trace $X_n(K)$. In this case, $H_2(X_n(K))\cong\mathbb{Z}$, $\alpha^{+}=1$, and $\alpha^{-}=0$. Thus, the formula in the preceding lemma reduces to
\[
s\bigl(X_n(K);1\bigr)
=
-\lim_{r\to\infty}
\left(
s\bigl(-K(1+r,r)\bigr)+2r
\right)
-w\bigl(K(1,0)\bigr).
\]
Under the usual framing convention, $K(1,0)$ is the $n$-framed knot $K$, so that $w\bigl(K(1,0)\bigr)=n$. Hence,
\[
s\bigl(X_n(K);1\bigr)
=
-\lim_{r\to\infty}
\left(
s\bigl(-K(1+r,r)\bigr)+2r
\right)-n.
\]
\end{remark}

Lemma~\ref{lem:ren-willis-formula} and Remark~\ref{rem:trace-formula} now give the tools needed for the applications stated in the introduction.

\begin{proof}[Proof of Corollary~\ref{cor:exotic-zero-traces}]
Let $K$ and $K'$ be $0$-framed oriented knots satisfying
\[
s(K)=2g_4(K)=\overline{\operatorname{sl}}(K)+1=2
\qquad\text{and}\qquad
s(K')<2,
\]
and suppose that there is an orientation-preserving homeomorphism from $X_0(K)$ to $X_0(K')$.
Assume that there exists an integer $M>0$ such that, for every $r\geq M$,
\[
s\bigl(K(2r+1,0)\bigr)
\leq
-s\bigl(-K(1+r,r)\bigr).
\]
Subtracting $2r$ from both sides gives
\[
s\bigl(K(2r+1,0)\bigr)-2r
\leq
-s\bigl(-K(1+r,r)\bigr)-2r.
\]
Applying the lower bound of Lemma~\ref{prop:s-framed-cable} with $a=2r+1$ and $n=0$, we obtain
\[
s\bigl(K(2r+1,0)\bigr)
\geq
(2r+1)s(K)-2r.
\]
Since $s(K)=2$, this becomes
\[
s\bigl(K(2r+1,0)\bigr)\geq 2r+2.
\]
Therefore, for every $r\geq M$,
\[
2
\leq
s\bigl(K(2r+1,0)\bigr)-2r
\leq
-s\bigl(-K(1+r,r)\bigr)-2r.
\]
Passing to the limit as $r\to\infty$, we obtain
\[
\lim_{r\to\infty}
\left(
-s\bigl(-K(1+r,r)\bigr)-2r
\right)
\geq 2.
\]
By Remark~\ref{rem:trace-formula}, it follows that
\[
s\bigl(X_0(K);1\bigr)\geq 2.
\]
On the other hand, the sequence appearing inside the limit in Lemma~\ref{lem:ren-willis-formula} is non-decreasing. Evaluating it at $r=0$ therefore gives
\[
s\bigl(X_0(K');1\bigr)
\leq
s(K')
<
2.
\]
Hence,
\[
s\bigl(X_0(K);1\bigr)
>
s\bigl(X_0(K');1\bigr).
\]
By the symmetry property of the lasagna $s$-invariant, $s(X;1)=s(X;-1)$ when $H_2(X)\cong\mathbb{Z}$; see Theorem~4.11(2) of \cite{RW25}. Any orientation-preserving diffeomorphism between the two zero traces would send a generator of second homology to a generator, hence to either $1$ or $-1$, and would preserve the corresponding lasagna $s$-invariant. The strict inequality above therefore rules out such a diffeomorphism. Together with the assumed orientation-preserving homeomorphism, this shows that $X_0(K)$ and $X_0(K')$ form an exotic pair of oriented zero traces.
\end{proof}

\begin{remark}
Note that, for a fixed knot $K$ and any $r\ge0$, Proposition 7.6(5) of \cite{MMSW23} gives
\[
s\bigl(K(1+r,r)\bigr)\leq -s\bigl(-K(1+r,r)\bigr).
\]
Moreover, as unoriented links, $K(2r+1,0)$ and $K(1+r,r)$ are the same. Suppose that there exists $M\geq0$ such that, for every $r\geq M$, the change of orientation from $K(2r+1,0)$ to $K(1+r,r)$ does not decrease the $s$-invariant; that is,
\[
s\bigl(K(2r+1,0)\bigr)\leq s\bigl(K(1+r,r)\bigr).
\]
Combining this with the inequality above, we obtain
\[
s\bigl(K(2r+1,0)\bigr)
\leq
s\bigl(K(1+r,r)\bigr)
\leq
-s\bigl(-K(1+r,r)\bigr).
\]
Hence, in order to verify the hypothesis of the preceding corollary, it is sufficient to check whether
\[
s\bigl(K(2r+1,0)\bigr)\leq s\bigl(K(1+r,r)\bigr)
\]
holds for all sufficiently large $r$.
\end{remark}

\begin{proof}[Proof of Proposition~\ref{prop:lasagna-bound}]
Let $K\subset S^{3}$ be an oriented knot equipped with framing $n$, and suppose that $K$ satisfies one of the hypotheses of Theorem~\ref{thm:main-s-cable}. Let $X_n(K)$ denote the $n$-trace of $K$, identify $H_2(X_n(K))\cong\mathbb{Z}$ using the standard generator, and fix an integer $\alpha>0$. By Lemma~\ref{lem:ren-willis-formula}, the sequence inside the limit is non-decreasing. Hence its limiting value is at least its value at $r=0$. Taking $L=\varnothing$, $\alpha^{+}=\alpha$, and $\alpha^{-}=0$, and using
\[
w\bigl(K(\alpha,0)\bigr)=n\alpha^{2},
\]
we obtain
\[
s\bigl(X_n(K);\alpha\bigr)
\leq
-s\bigl(-K(\alpha,0)\bigr)-n\alpha^{2}-\alpha+1.
\]

It remains to bound $-s\bigl(-K(\alpha,0)\bigr)$. In the cabling convention used throughout this paper, $K(\alpha,0)$ is the $(\alpha,n\alpha)$-cable link of $K$. Consequently, $-K(\alpha,0)$ is the corresponding $(\alpha,-n\alpha)$-cable link of $-K$. We use the band cobordisms described in Section~\ref{sec:s-cables}. Attaching $\alpha-1$ positive half-twisted bands gives a connected cobordism
\[
\Sigma\colon
-K(\alpha,0)
\longrightarrow
(-K)_{\alpha,-n\alpha+1}.
\]
Since
\[
(-K)_{\alpha,-n\alpha+1}
\cong
-\bigl(K_{\alpha,n\alpha-1}\bigr),
\]
we may equivalently regard the outgoing knot as the orientation-reversed mirror of $K_{\alpha,n\alpha-1}$. Moreover,
\[
-\chi(\Sigma)=\alpha-1.
\]
Therefore, Corollary~\ref{cor:s-cobordism} gives
\[
\left|
s\bigl(-\bigl(K_{\alpha,n\alpha-1}\bigr)\bigr)
-
s\bigl(-K(\alpha,0)\bigr)
\right|
\leq
\alpha-1.
\]
In particular,
\[
-s\bigl(-K(\alpha,0)\bigr)
\leq
\alpha-1-s\bigl(-\bigl(K_{\alpha,n\alpha-1}\bigr)\bigr).
\]
Using $s(-J)=-s(J)$ for every oriented knot $J$, we obtain
\[
-s\bigl(-K(\alpha,0)\bigr)
\leq
\alpha-1+s\bigl(K_{\alpha,n\alpha-1}\bigr).
\]
Applying the upper bound from Theorem~\ref{thm:main-s-cable}, we have
\[
s\bigl(K_{\alpha,n\alpha-1}\bigr)
\leq
\alpha\,s(K)+(\alpha-1)(n\alpha-1+1),
\]
and hence
\[
s\bigl(K_{\alpha,n\alpha-1}\bigr)
\leq
\alpha\,s(K)+n\alpha(\alpha-1).
\]
Consequently,
\[
-s\bigl(-K(\alpha,0)\bigr)
\leq
\alpha-1+\alpha\,s(K)+n\alpha(\alpha-1).
\]
Combining this bound with the formula from Lemma~\ref{lem:ren-willis-formula}, we obtain
\begin{align*}
s\bigl(X_n(K);\alpha\bigr)
&\leq
\alpha-1+\alpha\,s(K)+n\alpha(\alpha-1)-n\alpha^{2}-\alpha+1 \\
&=
\alpha\,s(K)-n\alpha.
\end{align*}
Therefore,
\[
s\bigl(X_n(K);\alpha\bigr)
\leq
\alpha\bigl(s(K)-n\bigr),
\]
as desired.
\end{proof}

\section[The $\tau$-Invariant of Coherently Oriented Cables]{The $\tau$-Invariant of Coherently Oriented Cables}
\label{sec:tau-cables}

We prove Theorem~\ref{tau-main-thm} by extending the knot-cabling
bounds of Lemma~\ref{lem:tau-cable} to coherently oriented
cable links. To do so, we modify Van Cott's function $h$ by
introducing a correction term depending on the number of cable
link components. Cavallo's one-sided cobordism inequalities establish
monotonicity of the corrected function, allowing us to compare
each cable with cables at larger and smaller coprime parameters.
Throughout this section, $\tau$ denotes Cavallo's link invariant,
which agrees with the classical Ozsváth--Szabó invariant on
knots; see \cite[Theorem~1.3]{CavalloTau}.

\begin{lemma}[Van Cott \cite{VanCott10}; Hom \cite{Hom2014}]
\label{lem:tau-cable}
Let \(K\subset S^{3}\) be a knot. Then, for every \(p>0\) and every integer \(q\) relatively prime to \(p\),
\[
p\,\tau(K)+\frac{(p-1)(q-1)}{2}
\leq
\tau(K_{p,q})
\leq
p\,\tau(K)+\frac{(p-1)(q+1)}{2}.
\]
Moreover, the following formulas hold:
\begin{enumerate}
\item If \(\tau(K)=g_{4}(K)>0\), then
\[
\tau(K_{p,q})
=
p\,\tau(K)+\frac{(p-1)(q-1)}{2}.
\]

\item If \(\tau(K)=-g_{4}(K)<0\), then
\[
\tau(K_{p,q})
=
p\,\tau(K)+\frac{(p-1)(q+1)}{2}.
\]

\end{enumerate}
\end{lemma}

In Corollary~3 of \cite{VanCott10}, Van Cott established the two formulas above under the hypotheses
\[
\tau(K)=g_{3}(K)
\qquad\text{and}\qquad
\tau(K)=-g_{3}(K),
\]
respectively. Hom's cabling formula in Theorem~1 of \cite{Hom2014}, together with Corollary~4 of the same paper, allows these hypotheses to be weakened. Hom's theorem applies to every pair of relatively prime cabling parameters, so no additional sign restriction on $q$ is needed here. In particular, if
\[
\tau(K)=g_{4}(K)>0,
\]
then $\varepsilon(K)=1$, whereas if
\[
\tau(K)=-g_{4}(K)<0,
\]
then $\varepsilon(K)=-1$. Consequently, the corresponding equality formulas remain valid with the Seifert genus $g_{3}(K)$ replaced by the smooth four-ball genus $g_{4}(K)$.

To extend the knot-cabling bounds beyond the coprime case,
we use Cavallo's link cobordism inequalities. We first recall
the bound from \cite[Proposition~4.7]{CavalloTau}, then
examine the bounds produced in its proof to identify the correction term needed for Van Cott's function $h$.

\begin{lemma}[Proposition~4.7 of \cite{CavalloTau}]
\label{lem:cavallo-cobordism}
Let $\Sigma$ be a smooth, oriented cobordism between links
$L_1$ and $L_2$, and suppose that every connected component
of $\Sigma$ meets both $L_1$ and $L_2$. Then
\[
\left|\tau(L_1)-\tau(L_2)\right|
\leq
g(\Sigma)+\max\{l_1(\Sigma),l_2(\Sigma)\}.
\]
Here, writing
\[
\Sigma=\Sigma_1\sqcup\cdots\sqcup\Sigma_J,
\]
we set
\[
g(\Sigma)=\sum_{k=1}^{J}g(\Sigma_k)
\]
and, for $i=1,2$,
\[
l_i(\Sigma)
=
\sum_{k=1}^{J}
\left(
\#\{\text{components of $L_i$ meeting $\Sigma_k$}\}-1
\right).
\]
\end{lemma}

In proving this bound, Cavallo establishes two inequalities
\begin{align}
\tau(L_2)
&\leq \tau(L_1)+g(\Sigma)+l_2(\Sigma),
\tag{A}\label{eq:cavallo-A}
\\
\tau(L_1)
&\leq \tau(L_2)+g(\Sigma)+l_1(\Sigma).
\tag{B}\label{eq:cavallo-B}
\end{align}
Together, they imply the estimate in
Lemma~\ref{lem:cavallo-cobordism}. Individually, however, they
retain the component count at each end of the cobordism.
This distinction is essential for the monotonicity argument
below and motivates the following modification of Van Cott's
function $h$.

\begin{definition}
\label{def:new-H}
Fix an oriented knot $K\subset S^3$ and a positive integer $p$.
For each $q\in\mathbb{Z}$, equip the cable $K_{p,q}$ with the
coherent orientation induced by $K$, and set $d_q=\gcd(p,q)$.
Define
\[
H\colon\mathbb{Z}\longrightarrow\frac{1}{2}\mathbb{Z},
\]
by
\[
H(q)=\tau(K_{p,q})-\frac{d_q-1}{2}-\frac{p-1}{2}\,q.
\]
Here, $\tau(K_{p,q})$ denotes Cavallo's link $\tau$-invariant.
When $\gcd(p,q)=1$, the correction term vanishes, and
$H(q)$ agrees with Van Cott's function $h(q)$ for the knot
$\tau$-invariant.
\end{definition}

The proof of Theorem~\ref{tau-main-thm} relies on the following monotonicity
result for the function $H$.

\begin{lemma}
\label{lem:vancott-monotonicity-3}
The function $H$ is non-increasing and bounded from below.
More precisely, for any integers $q>r$,
\[
-(p-1)\leq H(q)-H(r)\leq 0.
\]
\end{lemma}

\begin{proof}
The case $p=1$ is immediate since $H(q)=\tau(K)$ for every
$q\in\mathbb{Z}$. Assume $p>1$, and let $q,r\in\mathbb{Z}$
with $q>r$. Set
\[
d_q=\gcd(p,q),
\qquad
d_r=\gcd(p,r).
\]
Consider the oriented cobordism $\Sigma$ from $K_{p,q}$ to
$K_{p,r}$ constructed in the proof of Lemma~\ref{lem:vancott-monotonicity-2}.
This cobordism is connected and consists of $(p-1)(q-r)$
band attachments, so
\[
\chi(\Sigma)=-(p-1)(q-r).
\]
Since its boundary has $d_q+d_r$ components,
\[
g(\Sigma)
=
\frac{2+(p-1)(q-r)-d_q-d_r}{2}.
\]
Applying \eqref{eq:cavallo-B} with $L_1=K_{p,q}$ and $L_2=K_{p,r}$, we obtain
\begin{align*}
\tau(K_{p,q})-\tau(K_{p,r})
&\le g(\Sigma)+d_q-1\\
&=\frac{(p-1)(q-r)+d_q-d_r}{2}.
\end{align*}
Consequently,
\[
H(q)-H(r)
=
\tau(K_{p,q})-\tau(K_{p,r})
-\frac{d_q-d_r}{2}
-\frac{p-1}{2}(q-r)
\le 0.
\]
Thus, $H$ is non-increasing.

For the lower bound, take integers $Q>q$ and $R<r$
with $\gcd(p,Q)=\gcd(p,R)=1$. Since $H$ is non-increasing,
\[
H(r)-H(q)\le H(R)-H(Q).
\]
At the coprime parameters $Q$ and $R$, the correction terms
vanish, so Van Cott's inequality applies and gives
\[
H(R)-H(Q)\le p-1.
\]
Therefore,
\[
H(q)-H(r)\ge -(p-1).
\]
This completes the proof.
\end{proof}

We now have all the ingredients needed to prove Theorem~\ref{tau-main-thm}.

\begin{proof}[Proof of Theorem~\ref{tau-main-thm}]
Let $K\subset S^3$ be an oriented knot, fix a positive integer
$p$, and let $q\in\mathbb{Z}$. Set $d=\gcd(p,q)$, and choose
integers $Q>q$ and $R<q$, both relatively prime to $p$.
By Lemma~\ref{lem:vancott-monotonicity-3},
\[
H(Q)\le H(q)\le H(R).
\]
Since the correction terms vanish at $Q$ and $R$, the bounds
in Lemma~\ref{lem:tau-cable} give
\[
H(Q)
=
\tau(K_{p,Q})-\frac{p-1}{2}Q
\ge p\,\tau(K)-\frac{p-1}{2}
\]
and
\[
H(R)
=
\tau(K_{p,R})-\frac{p-1}{2}R
\le p\,\tau(K)+\frac{p-1}{2}.
\]
Consequently,
\[
p\,\tau(K)-\frac{p-1}{2}
\le H(q)
\le p\,\tau(K)+\frac{p-1}{2}.
\]
Substituting
\[
H(q)=\tau(K_{p,q})-\frac{d-1}{2}-\frac{p-1}{2}q
\]
and rearranging yields
\[
p\,\tau(K)+\frac{(p-1)(q-1)+d-1}{2}
\le \tau(K_{p,q})
\le p\,\tau(K)+\frac{(p-1)(q+1)+d-1}{2}.
\]

Suppose first that $\tau(K)=g_4(K)>0$.
Item~(1) of Lemma~\ref{lem:tau-cable} implies
\[
H(Q)=H(R)=p\,\tau(K)-\frac{p-1}{2}.
\]
Since $H(Q)\le H(q)\le H(R)$, it follows that
\[
H(q)=p\,\tau(K)-\frac{p-1}{2},
\]
and hence
\[
\tau(K_{p,q})
=
p\,\tau(K)+\frac{(p-1)(q-1)+d-1}{2}.
\]

Similarly, if $\tau(K)=-g_4(K)<0$, item~(2) of Lemma~\ref{lem:tau-cable} gives
\[
H(Q)=H(R)=p\,\tau(K)+\frac{p-1}{2}.
\]
Monotonicity therefore forces
\[
H(q)=p\,\tau(K)+\frac{p-1}{2},
\]
which is equivalent to
\[
\tau(K_{p,q})
=
p\,\tau(K)+\frac{(p-1)(q+1)+d-1}{2}.
\]
This completes the proof.
\end{proof}

\small
\bibliographystyle{amsplain}
\bibliography{references}

\normalsize

\noindent
\textsc{Department of Mathematics}\\
\textsc{University at Buffalo, State University of New York}\\
Buffalo, NY 14260, USA\\
\textit{Email:} aadhishw@buffalo.edu

\end{document}